\documentclass[11pt, bezier,amstex]{article}
\usepackage{amssymb,amsmath,latexsym, bbm, amsthm}

\usepackage{eucal,mathrsfs,dsfont}
\usepackage{color}

\newcommand{\R}{\mathds R}

\newcommand{\Pp}{\mathds P}
\newcommand{\Ee}{\mathds E}
\newcommand{\I}{\mathds 1}

\newcommand{\sA}{\mathcal A}
\newcommand{\EE}{\mathcal E}
\newcommand{\FF}{\mathcal F}
\newcommand{\LL}{\mathcal L}

\newcommand{\wt}{\widetilde}
\newcommand{\wh}{\widehat}

\newtheorem{theorem}{Theorem}[section]
\newtheorem{lemma}[theorem]{Lemma}
\newtheorem{proposition}[theorem]{Proposition}

\theoremstyle{definition}

\newtheorem{example}[theorem]{Example}

\newtheorem*{acknowledgement}{Acknowledgement}

\numberwithin{equation}{section}
\allowdisplaybreaks

\begin{document}

\title{\bf Ergodicity for  Time Changed Symmetric Stable Processes  }

\author{{\bf Zhen-Qing Chen}\thanks{Research supported in part
by NSF Grants DMS-1206276 and  NNSFC Grant 11128101.
 } \quad and \quad {\bf Jian Wang}\thanks{Research supported in part
 by  National Natural Science Foundation of
China (No.\ 11201073), the Program for New Century Excellent
Talents in Universities of Fujian (No.\ JA12053) and the Program for Nonlinear Analysis and Its Applications (No.\ IRTL1206).}}

\date{}

\maketitle

\begin{abstract} In this paper we  study the ergodicity and the related
semigroup property for a class of symmetric Markov jump processes
associated with time changed symmetric $\alpha$-stable processes.
For this purpose, explicit and sharp criteria for Poincar\'{e} type
inequalities (including Poincar\'{e}, super Poincar\'{e} and weak
Poincar\'{e} inequalities) of the corresponding non-local Dirichlet
forms are derived. Moreover, our main results, when applied to a
class
 of one-dimensional stochastic differential equations driven
 by symmetric $\alpha$-stable processes, yield sharp criteria for their various
  ergodic properties and corresponding functional inequalities.

\medskip


\end{abstract}

\section{Introduction and Main Results}\label{section1}

Given a conservative symmetric Markov process which is not ergodic
 such as Brownian motion or symmetric stable process on $\R^d$,
is it possible to turn it into an ergodic process by a time-change?
It is known that transience and recurrence are invariant under
time-change (see \cite[Theorem 5.2.5]{CF2}). Thus one can never turn
a transient process into a recurrent process. In fact, after a
time-change a transient conservative process may have finite
lifetime. In \cite{CF1}, extension of such a time changed process
has been investigated for transient reflected Brownian motion. In
this paper, we investigate various ergodic properties of recurrent
symmetric stable processes under suitable time-change.

For $\alpha\in(0,2)$, let $X$ be a rotationally symmetric
$\alpha$-stable process on $\R^d$. It is well known that its
infinitesimal generator is the fractional Laplacian
   $\Delta^{\alpha/2}:=-(-\Delta)^{\alpha/2}$ on $\R^d$,
   which enjoys the following expression
$$\Delta^{\alpha/2}u(x)=\int_{\R^d\setminus\{0\}}
\Big(u(x+z)-u(x)-\nabla u(x)\cdot z\I_{\{|z|\le
1\}}\Big)\,\frac{C_{d,\alpha}}{|z|^{d+\alpha}} dz.
$$
Here $C_{d,\alpha}=\frac{\alpha 2^{\alpha-1} \Gamma ((d+\alpha)/2)}
{\pi^{d/2} \Gamma (1-\alpha/2)}$ is the normalizing constant so that
the Fourier transform $\widehat {\Delta^{\alpha/2} u}(\xi)$ of
$\Delta^{\alpha/2} u$ is  $- |\xi|^\alpha \widehat u (\xi)$. Observe
that
\begin{equation*}
\lim_{\alpha \uparrow 2} \frac{C_{d, \alpha}}{2-\alpha}
 = \frac{d \, \Gamma (d/2)}{\pi^{d/2}}.
\end{equation*}
The Dirichlet form $(\EE, \FF)$ of $X$
on $L^2 (\R^d; dx)$ is given by
\begin{equation*}
\begin{split}
\EE (f, g)=&\, \frac12 \int_{\R^d \times \R^d} (f(x)-f(y))(g(x)-g(y)) \frac{C_{d, \alpha}}{|x-y|^{d+\alpha}} dx dy, \\
\FF =&\, \{ f\in L^2(\R^d; dx): \ \EE (f, f)<\infty \}.
\end{split}
\end{equation*}
Note that for $f\in C^2_c(\R^d)$,
$$ \EE (f, g) = - \int_{\R^d} g(x) \Delta^{\alpha/2} f (x) dx
\quad \hbox{for every } g\in \FF.
$$
It is known that $X$ is recurrent if and only if $\alpha \geq d$. In
other words, the symmetric $\alpha$-stable process $X$ is recurrent
if and only if $d=1$ and $\alpha \in [1, 2)$. In the latter case,
$X$ is recurrent with Lebesgue measure as its symmetrizing measure
but it does not have a stationary probability distribution. On the
other hand, we know that a time-change of $X$ does not change its
transience and recurrence (see \cite[Theorem 5.2.5]{CF2}) but will
change its symmetrizing measure. Let $a$ be a positive function on
$\R^d$ so that $1/a$ is $L^1(\R^d; dx)$ locally integrable. Let $\mu (dx)=
\frac{1}{a(x)} dx$ and $A_t:= \int_0^t 1/ a(X_s) ds$, which is a
positive continuous additive functional of $X$. Define
$$
\tau_t=\inf\{s>0: A_s>t\}
$$
and set $Y_t= X_{\tau_t}$. It is known (cf. \cite[Theorem
5.2.2]{CF2}) that $Y$ is a $\mu$-symmetric strong Markov process on
$\R^d$.

From now on, we assume $d=1$, $\alpha \in [1, 2)$ and $a$ is a positive
 and locally bounded measurable function on $\R$ so that $\mu$ defined above is a probability measure
 (that is, $\int_{\R}  a(x)^{-1} dx=1$).
As we noted above, the time changed process $Y$ is a pointwise
recurrent $\mu$-symmetric Markov process on $\R$ and so
every point in $\R$ has positive $\EE$-capacity.
 In particular, $\mu$ is a
reversible probability measure (and hence an invariant probability
measure) of $Y$. In fact, we have

\begin{proposition}\label{P:1.1}
 $\mu$ is the unique invariant measure  of $Y$
 and for every $x\in \R$ and $f\in C_b (\R)$,
 $$
 \lim_{t\to \infty} \frac1t \int_0^t \Ee_x \left[  f(Y_s) \right] ds
  = \int_{\R} f(x) \mu (dx).
 $$
\end{proposition}

The goal of this paper is to study the ergodicity of $Y$; that is, the way of
the marginal distribution of $Y$ converging to its equilibrium distribution $\mu$. For this, we need to describe the Dirichlet form of the time changed process
$Y$.

\medskip

Let $\FF_e$ be the extended Dirichlet space of $(\EE, \FF)$ for the
symmetric $\alpha$-stable process $X$. By \cite[(6.5.4)]{CF2},
\begin{equation}\label{e:1.2}
\FF_e = \{u: \hbox{ Borel measurable with } |u|<\infty \hbox{ a.e. and }
 \EE (u, u)<\infty\}.
 \end{equation}
It follows from \cite[Corollaries 3.3.6 and 5.2.12]{CF2} that
$(\FF_e, \EE)$ is also the extended Dirichlet space of the time changed process
$Y$. Thus the Dirichlet form for process $Y$ on $L^2(\R; \mu)$ is
$(\EE, \FF^\mu)$,
where
\begin{equation}\label{e:1.3}
\FF^\mu = \FF_e  \cap L^2 (\R; \mu) =\left\{ u\in L^2(\R; \mu): \EE
(u, u)<\infty \right\}.
\end{equation}
Since the space $C^\infty_c(\R)$ of smooth function with compact
support is a core for the Dirichlet form $(\EE, \FF)$ of $X$, by
\cite[Theorem 5.2.8]{CF2}, $C^\infty_c (\R)$ is also a  core for
$(\EE, \FF^\mu)$. It in particular implies that $(\EE, \FF^\mu)$ is
a regular Dirichlet form on $L^2(\R; \mu )$. By abusing the notation
a little bit, we also denote the $L^2$-generator of $X$ by
$\Delta^{\alpha/2}$. Then the $L^2$-generator of $Y$ (or
equivalently, of $(\EE, \FF^\mu)$) is  $  a \Delta^{\alpha/2}$.
Let $$P_tf(x)=\Ee_x(f(Y_t)), \quad f\in \mathscr{B}_b(\R)$$ be the semigroup of $Y$ (or
equivalently, associated with $(\EE, \FF^\mu)$).

\subsection{Main Results}

 We now present results on various ergodic properties of $Y$.
 For a function $f$ defined on $\R$, we use $\mu (f)$ to denote
 the integral $\int_{\R} f(x) \mu (dx)$ whenever it is well defined.
We first consider the case of $1<\alpha<2$.

\begin{theorem} \label{thm1}
Suppose $\alpha \in (1, 2)$. For  $r>0$, set
$$\Phi(r):=\inf_{|x|\ge r}\frac{a(x)}{(1+|x|)^\alpha},\qquad
\Phi_0(r):=\inf_{|x|\le r} \frac{a(x)}{(1+|x|)^\alpha}$$ and
$$
K(r):=\sup_{|x|\leqslant r}a(x)^{-1},\quad k(r):=\inf_{|x|\leqslant r}a(x)^{-1},
\quad K_0 (r) :=   \frac{K(r)^{1+1/\alpha}}{k(r)^2}.
$$
Assume in addition that
$a$ is locally bounded between two positive constants.
Then the following holds.

\begin{itemize}
\item[\rm (i)] If \begin{equation}\label{con1}
\lim_{r\to\infty}\Phi(r)>0,\end{equation} then the following Poincar\'{e} inequality
 \begin{equation}\label{thm1-1}\mu(f-\mu(f))^2\le C \EE(f,f), \quad f\in \FF^\mu \end{equation}
holds for some constant $C>0$. Equivalently,
for  every $f\in L^2(\R; \mu)$ and $t>0$,
\begin{equation}\label{e:1.6}
\|P_tf -\mu (f)\|_{L^2(\R;\mu)}\le e^{-t/C} \| f\|_{L^2(\R; \mu)}
.
\end{equation}

\item[\rm (ii)] If \begin{equation}\label{con2}
\lim_{r\to\infty}\Phi(r)=\infty,\end{equation} then the following
super Poincar\'{e} inequality
\begin{equation}\label{thm2-1}
\mu(f^2)\le r\EE (f,f)+\beta(r)\mu(|f|)^2, \quad r>0, f\in \FF^\mu \end{equation}
holds with
\begin{equation}\label{e:1.9} \beta(r)=
C_1\left(1+r^{-1/\alpha} K_0 \circ \Phi^{-1}(C_2/r)\right)
\end{equation}
 for some constants $C_1$ and $C_2>0$.
 Consequently, if  $\int_t^\infty  \beta^{-1}(r)/ r dr<\infty$ for some
 $t>0$, then
 there is a constant $C_3>0$ so that for every $f\in L^2(\R; \mu)$ and $t>1$,
 \begin{equation}\label{e:1.10}
\sup_{x\in \R} | P_t f(x)-\mu (f)|
 \leq C_3  e^{-t/C} \| f\|_{L^2(\R; \mu)},
\end{equation} where $C>0$ is a constant appearing in \eqref{e:1.6}.
In this case, we have
$$ \sup_{x\in \R} \| P_t(x, \cdot) -\mu \|_{TV}
:= \sup_{x\in \R}  \sup_{0\leq f \leq 1} | P_t f(x) -\mu (f) |
 \leq C_3  e^{-t/C},\quad t>1,
 $$ where for a signed measure $\nu$, $\| \nu \|_{TV}$ is denoted by its variation, and $P_t(x,\cdot)$ is the probabilty distribution of $Y_t$ with $Y_0=x$.

\item[\rm (iii)]If \begin{equation}\label{con3}
\lim_{r\to\infty}\Phi_0(r)=0,\end{equation} then the following weak
Poincar\'{e} inequality holds:
for every $r>0$,  there is
\begin{equation}\label{e:1.11}
\alpha(r)=C_4 \inf\Big\{\frac{1}{\Phi_0(s)}: \mu(B(0,s))\geq \frac{1}{1+r}\Big\}
\end{equation}
for some constant $C_4 >0$ independent of $r$ so that
\begin{equation}\label{thm3-1}\mu(f^2)\le \alpha(r)
\EE (f,f)+r\|f\|_\infty^2
\quad \hbox{for every } f\in \FF^\mu \hbox{ with } \mu(f)=0.
\end{equation}
Consequently,
\begin{equation*}
 \| P_t f -\mu (f) \|_{L^2(\R; \mu)}^2
\leq  \xi (t) \left( \| f \|_{L^2(\R; \mu)}^2+
\| f  \|_\infty^2\right)
\end{equation*}
for every $ f\in L^2(\R; \mu)$ and $t>0$, where
$ \xi (t) = 2\inf\{r>0: -  \alpha (r) \log r \leq 2 t\}$.
Note that $\lim_{t\to \infty} \xi (t)=0$.
\end{itemize}
\end{theorem}
As a direct consequence of \eqref{thm2-1}, we get that the following
defective Poincar\'{e} inequality
\begin{equation}\label{thm2-2}\mu(f^2)\le c_1
\EE(f,f)+c_2\mu(|f|)^2, \quad f\in \FF^\mu
\end{equation}
holds for some constants $c_1$, $c_2>0$. Since the Dirichlet form
$(\EE, \FF^\mu)$ is irreducible, i.e.\ $\EE(f, f)=0$ implies $f$ is
a constant function, it follows from \cite[Corollary 1.2]{W13} (see
also \cite[Theorem 1]{Mi}) that the defective Poincar\'{e}
inequality \eqref{thm2-2} is equivalent to the
 Poincar\'{e} inequality \eqref{thm1-1}. Therefore, the super
Poincar\'{e} inequality \eqref{thm2-1} is stronger than the
Poincar\'{e} inequality \eqref{thm1-1}.
Clearly, for $t>1$ the uniform ergodicity
\eqref{e:1.10} is also stronger than ergodicity \eqref{e:1.6}
under stationary distribution.
The super Poincar\'e
inequality \eqref{thm2-1} is equivalent to the uniform integrability
of the semigroup $(P_t)_{t\ge0}$, and also the absence of the
essential spectrum of its generator if the semigroup
$(P_t)_{t\ge0}$ has an asymptotic density, see, e.g., \cite[Theorems 6.1, 2.1 and 5.1]{W2000}.
It often implies the ultracontractivity of $P_t$ for some $t>0$.
See \cite[Chapter 3]{WBook} for more details about the
applications of super Poincar\'{e} inequality.
Clearly the Poincar\'{e} inequality \eqref{thm1-1} implies that the weak
Poincar\'{e} inequality \eqref{thm3-1} holds with $\alpha (r)=C$.
On the other hand, under condition \eqref{con3},
it is easy to see that $\Phi(r)=0$ for all $r>0$,
and the function $\alpha(r)$ defined in \eqref{e:1.11} tends to $\infty$ as $r\to 0$.
 The weak Poincar\'e
inequality \eqref{thm3-1} characterizes the $L^2$-convergence rates
of the semigroup $(P_t)_{t\ge0}$ slower than exponential.
See, e.g.,  \cite[Chapter 4]{WBook} for more information on weak Poincar\'{e} inequality and its consequences.

\medskip

Theorem \ref{thm1} is sharp in many situations;  see Examples
\ref{exa1} and \ref{exa2} below. Note that the inequalities \eqref{thm1-1} and \eqref{thm2-1} are related to the  weighted Poincar\'{e} inequalities for non-local Dirichlet forms studied in \cite{CW12}. However, applying
\cite[Proposition 1.7]{CW12} with $d=1$ to our case,
one only gets that the Poincar\'{e} inequality \eqref{thm1-1} holds when
$$
\lim_{|x|\to\infty}\frac{a(x)}{(1+|x|)^{1+\alpha}}>0,
$$
 which is stronger than condition \eqref{con1} and is far from being optimal.
 See also Examples \ref{exa1} and \ref{exa2} below. Furthermore, according to
\cite[Table 5.1, p.\ 100]{Chen}, Theorem \ref{thm1} is also optimal
for the case when $\alpha=2$, which corresponds to the limiting
 one-dimensional Brownian case. See Appendix for more details on
 this Brownian case.

We now present some examples to illustrate our main results.
The proofs of the claims made in these examples are
given in Section 3.

\begin{example}\label{exa1}
Suppose $\alpha \in (1, 2)$.
Let $a(x)=C_\gamma(1+|x|)^{\gamma}$ with $\gamma>1$ such that $\int a(x)^{-1}\,dx=1$.
\begin{itemize}
\item[(i)] The Poincar\'{e} inequality \eqref{thm1-1} holds for some
constant $C>0$ if and only if $\gamma\ge \alpha$.
\item[(ii)] The super Poincar\'{e} inequality \eqref{thm2-1} holds for
some function $\beta:(0,\infty)\to(0,\infty)$ if and only if
$\gamma>\alpha$. In this case, there exists a constant $c>0$
such that the super Poincar\'{e} inequality \eqref{thm2-1} holds with
$$
\beta(r)=  c \Big(1+r^{-(\frac{1}{\alpha}+\frac{2\gamma}{\gamma-\alpha})}\Big),\quad
r>0,
$$
which is equivalent to
\begin{equation}\label{e:1.15}
\|P_t\|_{L^1(\R;\mu)\to L^\infty(\R;\mu)}\le
c_0\left(1+t^{-(\frac{1}{\alpha}+\frac{2\gamma}{\gamma-\alpha})}\right),\quad
t>0
\end{equation}
 for some constant $c_0>0$.
Since $\beta^{-1}(r)/r$ is integrable at infinity, uniform strong ergodicity \eqref{e:1.10} holds.

\item[(iii)] If $\gamma\in(1,\alpha)$, then the weak Poincar\'{e}
inequality \eqref{thm3-1} holds with
$$\alpha(r)=c\Big(1+r^{-(\alpha-\gamma)/(\gamma-1)}\Big),\quad r>0$$
 for some
constant $c>0$. Consequently, there exists a constant $c_0>0$
such that
$$\|P_t-\mu\|_{L^\infty(\R;\mu)\to L^2(\R;\mu)}\le
c_0 t^{-(\gamma-1)/({\alpha-\gamma})},\quad t>0.$$
The
rate function $\alpha$ given above is sharp in the sense that \eqref{thm3-1} does
not hold if
$$\lim_{r\to0}r^{(\alpha-\gamma)/(\gamma-1)}\alpha(r)=0.$$
\end{itemize}\end{example}

\begin{example}\label{exa2}
Suppose $1<\alpha<2$. Let $a(x)=C_{\alpha,\gamma}(1+|x|)^{\alpha}\log^\gamma(e+|x|)$ with $\gamma\in\R$ such that $\int a(x)^{-1}\,dx=1$
\begin{itemize}
\item[(i)] The Poincar\'{e} inequality \eqref{thm1-1} holds for some
constant $C>0$ if and only if $\gamma\ge 0$.

\item[(ii)] The super Poincar\'{e} inequality \eqref{thm2-1} holds for
some function $\beta:(0,\infty)\to(0,\infty)$ if and only if
$\gamma>0$. In this case there exists a constant $c>0$
such that the super Poincar\'{e} inequality \eqref{thm2-1} holds with
$$\beta(r)=  \exp\Big(c(1+r^{-1/\gamma})\Big),\quad
r>0.
$$
Consequently,  when $\gamma>1$,
\begin{equation}\label{e:1.16}
\|P_t\|_{L^1(\R;\mu)\to L^\infty(\R;\mu)}\le
\exp\Big(c_0(1+t^{-1/(\gamma-1)})\Big),\quad t>0
\end{equation}
holds for some constant $c_0>0$.
Moreover, the uniform strong ergodicity \eqref{e:1.10} holds when $\gamma>1$.
The rate function $\beta$ above is
sharp in the sense that \eqref{thm2-1} does not hold if
$$\lim_{r\to0} r^{1/\gamma}\log \beta(r)=0.$$

In particular, the following log-Sobolev inequality
\begin{equation}\label{logine}\mu(f^2\log f^2)\le C\EE(f,f) \quad  \hbox{for } f\in \FF^\mu \hbox{ with } \mu(f^2)=1
\end{equation}
holds for some constant $C>0$ if and only if $\gamma\ge1$. In
this case, we have
\begin{equation}\label{e:1.17}
{\rm Ent}_\mu (P_tf)\le {\rm Ent}_\mu (f)e^{-4t/C},\quad t>0, f\in L^2(\R^d; \mu) \hbox{ with } f>0,
\end{equation}
 where
${\rm Ent}_\mu (f)=\mu(f\log f)-\mu(f)\log \mu(f).$

\item[(iii)]If $\gamma<0$, then the weak Poincar\'{e}
inequality \eqref{thm3-1} holds with
$$\alpha(r)=c\Big(1+\log^{-\gamma}(1+r^{-1})\Big),\quad r>0$$ for some
constant $c>0$. Consequently, there exist constants $c_1$ and $c_2>0$
such that
$$\|P_t-\mu\|_{L^\infty(\R;\mu)\to L^2(\R;\mu)}\le
\exp\big(c_1-c_2t^{1/(1-\gamma)}\big),\quad t>0.$$ The
rate function $\alpha$ defined above is sharp in the sense that \eqref{thm3-1} does
not hold if
$$\lim_{r\to0}\log^{\gamma}(1+r^{-1})\alpha(r)=0.$$
\end{itemize}\end{example}

\medskip

We now turn to the case of $\alpha=1$. For $1<\delta<\beta$ and $r>0$, define
$$
\Psi_\beta(r):=\inf_{|x|\ge r}\frac{a(x)}{(1+|x|)^\beta},\quad
 \Psi_{\beta,\delta}(r):=\Psi_\beta(r)(1+r)^{\beta-\delta},\quad K_0(r):=\left(\frac{K(r)}{k(r)}\right)^2.
$$

\begin{theorem}\label{thm2}
 Suppose that $\alpha =1$.
 If there exists a constant $\beta>1$ such that
\begin{equation}\label{thm2-1-1}
\lim_{r\to\infty}\Psi_\beta(r)>0,
\end{equation}
then for every $\delta \in (1, \beta)$ there are positive
constants $C_1$ and $C_2$  so that the  super Poincar\'{e}
inequality \eqref{thm2-1} holds with
$$
\beta(r)=C_1\left(1+r^{-1}K_0\circ \Psi^{-1}_{\beta,\delta}(C_2/r)\right).$$
In particular, the Poincar\'{e} inequality \eqref{thm1-1} holds with
some constant $C>0$; or equivalently, the process $Y$ is
$L^2(\R;\mu)$-exponentially ergodic, i.e. for every $f\in L^2(\R;\mu)$ and $t>0$,
$$\|P_tf-\mu( f)\|_{L^2(\R;\mu)}\le e^{-t/C}\|f\|_{L^2(\R;\mu)}.$$
\end{theorem}

To illustrate Theorem \ref{thm2}, we reconsider Example \ref{exa1} for the case of $\alpha=1$.

\begin{example}[{\bf Continuation of Example \ref{exa1} with $\alpha=1$}]\label{exa3}Let $\alpha=1$, and $a(x)=C_\beta(1+|x|)^{\gamma}$ with $\gamma>1$ such that $\int a(x)^{-1}\,dx=1$.
Then, for the Dirichlet form $(\EE, \FF^\mu)$, the super
Poincar\'{e} inequality \eqref{thm2-1} holds with
$$\beta(r)\le c \Big(1+r^{-(1+\frac{2\gamma}{\gamma-\delta})}\Big),\quad
r>0$$ for any $\delta\in(1,\gamma)$ and some constant $c=c(\delta)>0$;
and equivalently,
$$\|P_t\|_{L^1(\R;\mu)\to L^\infty(\R;\mu)}\le
c_0\left(1+t^{-(1+\frac{2\gamma}{\gamma-\delta})}\right),\quad
t>0$$ holds for some constant $c_0=c_0(\delta)>0$.
Consequently, the uniform strong ergodicity \eqref{e:1.10} holds.
\end{example}

\subsection{Application: one-dimensional SDEs driven by  symmetric stable processes}

As a direct application of our main results, we   consider the following one dimensional stochastic differential equation (SDE) driven by a symmetric
$\alpha$-stable process $X$ on $\R$:
\begin{equation}\label{sde1} dZ_t=\sigma (Z_{t-})\,dX_t,\end{equation}
where $\alpha \in (1, 2)$ and $\sigma:\R\to\R$ is a locally $1/\alpha$-H\"older
continuous and strictly positive function. Then, the SDE
\eqref{sde1} has a unique strong solution (see the proof of Theorem
\ref{pro100} in Section \ref{section3}). Denote by $Z$
this unique solution. The process $Z$ has  strong
Markov property, and its infinitesimal generator is given by
$$
\LL u(x)=\int_{\R\backslash\{0\}}\big(u(x+\sigma (x)z)-u(x)-u'(x)\sigma (x)z\I_{\{|z|\le 1\}}\big)\frac{C_{1,\alpha}}{|z|^{1+\alpha}}\,dz.$$ Letting $v=\sigma (x)z$, we find that
\begin{equation}\label{e:1.13}
\LL u(x)=\sigma (x)^\alpha
\int_{\R\backslash\{0\}}\big(u(x+v)-u(x)-u'(x)v\I_{\{|v|\le 1\}}\big)\frac{C_{1,\alpha}}{|v|^{1+\alpha}}\,dv
= \sigma(x)^\alpha\Delta^{\alpha/2}.
\end{equation}

According to Theorem \ref{thm1}(i), we have the following statement for the exponential ergodicity of the process $Z$.

\begin{theorem}\label{pro100}
Let $Z$ be the unique solution to the SDE \eqref{sde1}.
\begin{itemize}
\item[\rm (i)] Suppose that
\begin{equation}\label{pro1000}
\liminf_{|x|\to\infty} \frac{\sigma (x)}{|x|}>0.
\end{equation}
Then the process $Z$ is exponentially ergodic. More explicitly,   $\mu(dx):=\frac{\sigma(x)^{-\alpha}\,dx}{\int \sigma (x)^{-\alpha}\,dx}$ is a unique reversible probability measure for the process $Z$, and there is a constant $\lambda_0>0$  such that for any $x\in\R$
$$\|P_t(x,\cdot)-\mu\|_{\small TV}\le C(x)e^{-\lambda_0 t},$$ where $P_t(x,\cdot)$ is the probability distribution of $Z_t$ with $Z_0=x$ and $C(x)$ is a positive measurable function on $\R$.

\item[\rm (ii)] Furthermore,  if \begin{equation}\label{pro2000}
\liminf_{|x|\to\infty} \frac{\sigma (x)}{|x|^\gamma}>0.
\end{equation}  holds for some constant $\gamma>1$, then the process $Z$ is uniformly strongly ergodic, i.e. there are costants $C_1$ and $\lambda_1>0$ such that  $$\sup_{x\in \R} \|P_t(x,\cdot)-\mu\|_{\small TV}\le C_1e^{-\lambda_1 t}.$$
\end{itemize}
\end{theorem}

To the best of our knowledge, Theorem \ref{pro100} is new.
Recall that the symmetric $\alpha$-stable process $X$
is recurrent for $\alpha>1$, but it is not ergodic for any $\alpha>1$.
Theorem \ref{pro100} shows that the solution of a stochastic
differential equation driven by $X$  can be
exponentially ergodic under some growth condition such as \eqref{pro1000} on the multiplicative coefficient, and can be uniformly strongly ergodic under \eqref{pro2000}.
On the other hand, according to Theorem \ref{thm2}(i) and (the proof of) Theorem \ref{pro100}(i), Theorem \ref{pro100}(i) also holds for $\alpha=1$ if the condition \eqref{pro1000} is replaced by \eqref{pro2000}.

\medskip

The remainder of this paper is organized as follows. In
Section \ref{S:2}, we   give a proof of Proposition \ref{P:1.1} and
present some preliminaries of non-local Dirichlet
forms corresponding to time-change of pure jump symmetric L\'{e}vy
processes. The proofs of all the results and examples mentioned above
are presented in Section \ref{section3}, where the sharp drift condition for
truncated fractional Laplacian is also given.

\section{L\'evy processes and their time-change}\label{S:2}

We begin with a proof of Proposition \ref{P:1.1}.

\medskip

\noindent{\bf Proof of Proposition \ref{P:1.1}.}
 Observe that since $\alpha \in [1, 2)$,
every singleton is non-polar for symmetric $\alpha$-stable process
$X$ on $\R$, and hence for $Y$. So it follows from \cite[Lemma
3.5.5(ii)]{CF2}
 that any non-negative invariant function $h$ (that is, $P_t h = h$ for
 every $t>0$) is constant. This implies that the invariant $\sigma$-algebra
 ${\mathcal I}:=\{A: \I_A \hbox{ is an invariant function}\}$ is trivial.
So by \cite[Theorem 1.1]{Fi} (taking $g=1$ there),
$$
 \lim_{t\to \infty} \frac1t \int_0^t P_s f(x) ds = \int_{\R} f(y) \mu(dy)
$$
 for every bounded non-negative function $f$ on $\R$ and for
  every $x\in \R$.
Suppose $\nu$ is an invariant probability measure of $Y$; that is, $\nu = \nu P_t$ for every $t>0$. Integrating the above display with respect to $\nu$
and applying bounded convergence theorem, we get
$$ \int_{\R} f(y) \nu (dy) = \int_{\R} f(y) \mu (dy),
$$
from which we conclude $\nu = \mu$. \qed

\medskip

Suppose that $\rho$ is a nonnegative measurable function on $\R$ such that
$\rho(0)=0$, $\rho(z)=\rho(-z)$ and $\int (1\wedge z^2)\rho(z)\,dz<\infty.$
Let $\wt X$ be the symmetric pure jump L\'evy process on $\R$
with L\'evy measure $\rho (z) dz$, and denote by $(\wt \EE, \wt \FF)$ its
Dirichlet form on $L^2(\R; dx)$. By \cite[(2.2.16)]{CF2},
$$
\wt \EE (f, g)=\frac12 \int_{\R \times \R} (f(x)-f(y))(g(x)-g(y))
 \rho (x-y) dx dy
$$
and
$$
\wt \FF=\{ f\in L^2(\R; dx): \wt \EE (f, f)<\infty\}.
$$
Moreover, we know from \cite[\S 2.2.2]{CF2}, $(\wt \EE, \wt \FF)$ is a
regular Dirichlet form on $L^2(\R; dx)$.
Using smooth mollifier, we see that both  $C^\infty_c (\R)$
and $C^2_c (\R)$ are cores for $(\wt \EE, \wt \FF)$.
Denote by  $(\sA, {\mathcal D} (\sA))$ the $L^2$-generator of $\wt X$
(or equivalently, of $(\wt \EE , \wt \FF)$).
It is easy to verify that $ C^2_c (\R)\subset
{\mathcal D} (\sA )$ and for $u\in C^2_c(\R)$,
\begin{equation}\label{e:2.15}
\sA u(x)=\int_{\R \setminus \{ 0\}} \left(
u(x+z)-u(x)-u'(x)z\I_{\{|z|\le 1\}}\right)\rho(z)dz.
\end{equation}
Observe that $\sA u\in L^2(\R; dx)$ for $u\in C^2_c (\R)$.
This is because if we use $\widehat f$ to denote the Fourier transform of $f$,
then $\widehat {\sA u} (\xi)= \psi (\xi)\wh u (\xi) $,
where
$$\psi (\xi)= \int_{\R} (1-\cos (\xi y)) \rho (y) dy
\leq  2\int_{\R} (1\wedge (\xi y)^2) \rho (y) dy \leq  2(1+\xi^2)
\int_{\R} (1\wedge y^2 ) \rho (y) dy\leq c_0 (1+\xi^2).
$$
It follows that $\| \sA u\|_2^2 =\| \wh {\sA u}\|_2^2 \leq c_0^2 \|
(1+|\cdot|^2) \wh{u}\|_2^2 \le 2c_0^2( \| u\|_2^2 + \|
u''\|_2^2)<\infty$.

Let $a$ be a positive and locally bounded measurable function on
$\R$ so that $a(x)^{-1}$ is locally integrable. Set
$\mu(dx)=a(x)^{-1}\,dx$. Then $\mu$ is a smooth measure. Let $A_t:=
\int_0^t 1/ a(\wt X_s) ds$, which is a positive continuous additive
functional of $\wt X$. Define
$$
\tau_t=\inf\{s>0: A_s>t\}
$$
and set $\wt Y_t= \wt X_{\tau_t}$.
It follows from \cite[Theorems 5.2.2 and 5.2.8]{CF2}
that $\wt Y$ is a $\mu$-symmetric strong Markov process on $\R^d$,
whose associated Dirichlet form $(\wt \EE, \wt \FF^\mu)$ is
regular on $L^2(\R; \mu)$ having $C^2_c(\R)$ as its core.
In view of \cite[Corollary 5.2.12]{CF2}, $\wt \FF^\mu$ is given by
$\wt \FF^\mu = \wt \FF_e\cap L^2(\R; \mu)$, where
$\wt \FF_e$ is the extended Dirichlet space of $(\wt \EE, \wt \FF)$.
We know from \cite[(2.2.18) and (6.5.4)]{CF2} that
\begin{equation}\label{e:2.16}
\wt \FF_e:= \left\{ u\in L^1_{loc} (\R)\cap {\mathcal S}':
 \wt \EE (u, u)<\infty \right\}
 \end{equation}
when $\wt X$ is transient, and
\begin{equation}\label{e:2.17}
\wt \FF_e = \{u: \hbox{ Borel measurable with } |u|<\infty \hbox{ a.e. and }
\wt  \EE (u, u)<\infty\}
 \end{equation}
 when $\wt X$ is recurrent.
Here ${\mathcal S}'$ denotes the space of tempered distributions.
Clearly  the $L^2$-infinitesimal generator of $Y$ is
$\LL=a \sA$. Abusing the notation a little bit, we also use
$\sA u $ to denote the function defined by \eqref{e:2.15} whenever
 its right hand side is well defined. The same remark applies to operator
 $\LL = a \sA$.

\begin{proposition}\label{pro1}
 Assume that $\rho(z)=0$ for any $z\in\R$ with $|z|\le 1$, and that there exists a constant $\delta\in(0,1]$ such that
$\int_{\{|z|>1\}}|z|^\delta\rho(z)\,dz<\infty$. Set
\begin{equation*}
\begin{split}\mathscr{H}_\delta:=\Big\{ \varphi\in \mathscr{B}(\R): &\hbox{ there is some } C>0 \hbox{ such that } \\
 &|\varphi(x)-\varphi(y)|\le C|x-y|^{\delta } \hbox{ for any }
 x,y\in \R \hbox{ with } |x-y|\ge1\Big\}.
 \end{split}
 \end{equation*}
Then   $\LL \varphi$
exists pointwise as
a locally bounded function for $\varphi \in \mathscr{H}_\delta$.
Moreover for any bounded measurable function $f$ on $\R$ with compact support
 and for $\varphi\in \mathscr{H}_\delta$,
$$
 \int_{\{|x-y|> 1\}}
 |f(x)-f(y)|\, |\varphi(x)-\varphi(y)| \rho (x-y) dx dy <\infty
 $$
and
$$
-\int_{\R}  f (x) \LL\varphi (x) \mu (dx) =\frac{1}{2} \int_{\{|x-y|>1\}}
  ((f(x)-f(y)) (\varphi(x)-\varphi(y)) \rho (x-y) dx dy .
$$
 \end{proposition}

\begin{proof} Since $\rho(z)=0$ for all $|z|\le 1$,
$$\sA u(x)=\int_{\{|z|>1\}}(u(x+z)-u(x)) \rho(z) dz.$$
 Let $\varphi\in \mathscr{H}_\delta$.
Then
\begin{equation*}
\begin{split}
|\sA \varphi(x)|&\le \int_{\{|z|>1\}}\big|\varphi(x+z)-\varphi(x)\big|\,\rho(z)\,dz\\
&\le C\int_{\{|z|>1\}}|z|^\delta\rho(z)\,dz=:c <\infty,
 \end{split}\end{equation*}
so $\sA \varphi $ is well-defined pointwise on $ \R$ and is bounded.
Consequently, $\LL \varphi(x)= a(x) \sA \varphi(x)$ is poinwisely
well-defined on $\R$ and is locally bounded.
It follows that for every bounded function $f$ with compact support on $\R$,
due to the symmetry of $\rho (z)$,
\begin{equation*}
\begin{split} \int_{\{|x-y|> 1 \}}&
  |f(x)-f(y)| |\varphi(x)-\varphi(y)| \rho (x-y) dx dy \\
\leq& 2 \int_{\R} |f(x)| \left( \int_{\{y\in \R: |y-x|> 1\}}
 |\varphi(y)-\varphi(x)| \rho (x-y) dy \right) dx\\
=&  2 \int_{\R} |f(x)| \left( \int_{\{z\in \R: |z|>1\}}
 |\varphi(x+z)-\varphi(x)| \rho (z) dz\right) dx <\infty.
 \end{split}\end{equation*}
Again by the symmetry of $\rho (z)$,
\begin{equation*}
\begin{split}-\int f \LL\varphi\,d\mu&=-\int f(x) \sA\varphi(x)\,dx\\
&=-\int f(x)\,dx\int_{\{|y-x|> 1\}}(\varphi(y)-\varphi(x))\rho(x-y)\,dy\\
&=\frac{1}{2}\int_{\{|y-x| > 1\}}(f(x)-f(y))(\varphi(x)-\varphi(y))\rho(x-y)\,dy\,dx .
\end{split}
\end{equation*} The proof is completed.
\end{proof}

\section{Proofs}\label{section3}

Let $\alpha\in(0,2)$ and $d=1$.
We first consider the symmetric L\'evy process  $\wh X$ whose L\'evy measure
is   $ C_{1,\alpha} |z|^{-(1+\alpha)} \I_{\{ |z|>1\}} dz$.
Denote by $(\wh \EE, \wh \FF)$ the Dirichlet form of $\wh X$ on $L^2(\R; dx)$.
We know from Section 2 that
\begin{equation*}\label{e:3.1}
 \wh \EE (f, f)= \frac12 \int_{\R \times \R} (f(x)-f(y))^2
    \frac{C_{1, \alpha}}{|x-y|^{1+\alpha}} \I_{\{ |x-y|>1\}} dx dy
\end{equation*}
and
\begin{equation*}\label{e:3.2}
 \wh \FF =\left\{f\in L^2(\R; dx): \wh \EE (f, f)<\infty \right\}.
\end{equation*}
Moreover, it follows from \eqref{e:2.16} and \eqref{e:2.17} that the extended
Dirichlet space $\wh \FF_e$ of $(\wh \EE, \wh \FF)$ is given by
\begin{equation*}\label{e:3.3}
 \wh \FF_e=\left\{f\in L^1_{loc}(\R; dx): \wh \EE (f, f)<\infty \right\}.
\end{equation*}
Suppose that $a$ is a positive and locally bounded measurable function on
$\R$ such that  $\int_{\R} a(x)^{-1} dx=1$.
Define $\mu(dx):=a(x)^{-1}\,dx$.
Let $\wh Y_t=\wh X_{\tau_t}$ be the time-change of $\wh X$, where
$$
\tau_t := \inf \Big\{ s>0: \int_0^s a(\wh X_r)^{-1} dr >t \Big\}.
$$
We know from Section 2 that the time changed process $Y$ is $\mu$-symmetric
and its associated Dirichlet form is given by $(\wh \EE, \wh \FF^\mu)$,
 which is regular on $L^2(\R; dx)$ with core $C^2_c (\R)$. Here
\begin{equation*}\label{e:3.4}
 \wh \FF^\mu = \wh \FF_e \cap L^2(\R; \mu)
  =\left\{f\in L^2(\R; \mu): \wh \EE (f, f)<\infty \right\}.
\end{equation*}
Let $(\Delta^{\alpha/2}_{>1}, {\mathcal D} (\Delta^{\alpha/2}_{>1}))$ and
$(\LL^{(\alpha)}_{>1}, {\mathcal D} (\LL^{(\alpha)}_{>1}))$ be the $L^2$-infinitesimal generator
of $\wh X$ and $\wh Y$, respectively. It follows from Section 2 that
$\LL^{(\alpha)}_{>1}= a \Delta^{\alpha/2}_{>1}$ and $C^2_c(\R) \subset
{\mathcal D} (\Delta^{\alpha/2}_{>1}) \cap {\mathcal D} (\LL^{(\alpha)}_{>1})$.
Moreover, for $u\in C^2_c(\R)$,
$$
\Delta^{\alpha/2}_{>1}u(x):=\int_{\{|z|>1\}}
\big(u(x+z)-u(x)\big)\,\frac{C_{1,\alpha}}{|z|^{1+\alpha}} dz.
$$

\medskip

To prove Theorem \ref{thm1}(i), we need two lemmas. We begin with
 a local Poincar\'{e} inequality for $\EE$.
We use $\oint_A f(x) \mu (dx)$ to denote $\int_A f(x) \mu (x)/\mu (A)$.

\begin{lemma}\label{lemma1}For any $r>0$ and $f\in C_c^{2}(\R)$,
\begin{equation*}\label{lemma1-1}
 \int_{B(0,r)} \bigg(f(x)-\oint_{B(0,r)} f(x)\,\mu(dx) \bigg)^2\,\mu(dx)\\
 \leq \frac{2^\alpha K(r)^2 \, r^\alpha}{C_{1, \alpha} \, k(r)}\, \EE (f, f).
\end{equation*}
\end{lemma}

\begin{proof} For any $r>0$ and $f\in C_c^{2}(\R)$, by the Cauchy-Schwarz inequality,
\begin{equation*}
\begin{split}
\int_{B(0,r)} &\bigg(f(x)- \oint_{B(0,r)} f(y)\,\mu(dy) \bigg)^2\,\mu(dx)\\
&\leq  \int_{B(0,r)} \bigg( \oint_{B(0, r)} (f(x)- f(y))^2 \mu(dy) \bigg)\,\mu(dx)\\
&\leq  \frac{2^{1+\alpha}r^{1+\alpha}}{C_{1, \alpha}\, \mu (B(0, r))} \int_{B(0, r)\times B(0, r)} {(f(x)-f(y))^2} \frac{C_{1, \alpha}}{|x-y|^{1+\alpha}} \,
a(x)^{-1} a(y)^{-1} dx dy \\
&\leq \frac{2^\alpha K(r)^2 \, r^\alpha}{C_{1, \alpha} \, k(r)} \, \EE (f, f).
\end{split}
\end{equation*}
 \end{proof}

In the following lemma, we set $V(x)=1+|x|^\theta$ for $\theta\in (0,1)$.
It is clear that $V\in \mathscr{H}_\theta$,
where $\mathscr{H}_\theta$ is defined in Proposition \ref{pro1}.

\begin{lemma}\label{lemma2}
For any $\alpha\in (1, 2)$, there exists a constant $\theta\in (0,1)$ small enough such that for the function $V$ defined above, $\LL^{(\alpha)}_{>1}V$ is well defined and there exists a constant $r_0>0$ large enough such that for all $x\in\R$,
\begin{equation}\label{lemma2-1}
\LL^{(\alpha)}_{>1}V(x)\le \frac{\theta \pi C_{1,\alpha} \cot (\alpha\pi/2)}{2\alpha}\frac{a(x)}{(1+|x|)^\alpha}V(x)\I_{B(0,r_0)^c}(x)+\frac{2C_{1,\alpha}\sup_{|x|\le r_0}a(x)}{\alpha-\theta}\I_{B(0,r_0)}(x).
\end{equation}
\end{lemma}

\begin{proof} Let $\theta\in (0,1)$ be a constant which is determined later. First, according to the fact that
$|x+z|^\theta\le |x|^\theta+|z|^\theta$ for any $x$, $z\in \R,$ we have  \begin{equation*}\label{prooflemma2-0}
\begin{split}
|\Delta^{\alpha/2}_{>1}V(x)|\leq &\int_{\{|z|\ge1\}}\big| V(x+z)-V(x)\big|\frac{C_{1,\alpha}}{|z|^{1+\alpha}}\,dz\\
\le&\int_{\{|z|\ge1\}} |z|^\theta\frac{C_{1,\alpha}}{|z|^{1+\alpha}}\,dz=\frac{2C_{1,\alpha}}{\alpha-\theta}<\infty. \end{split}
\end{equation*} Then, by the local boundedness of $a$, we know that $\LL^{(\alpha)}_{>1}V(x)$ is a well defined and locally bounded measurable function. Hence, it suffices to  verify \eqref{lemma2-1} for $|x|$ large enough.

For this, one can follow the proof of \cite[Theorem 1.2]{Sa} (see \cite[(3.33), (3.35)--(3.38)]{Sa} for more details) and obtain that
\begin{equation}\label{prooflemma2-1}\limsup_{|x|\to\infty} |x|^{\alpha-\theta}\Delta^{\alpha/2}_{>1}V(x)=\frac{\theta C_{1,\alpha}}{\alpha}E(\alpha,\theta),\end{equation}
where
\begin{equation*}
\begin{split}E(\alpha,\theta):=&\frac{\alpha}{\theta}\sum_{i=1}^\infty \textsf{C}_\theta^{2i}\frac{2}{2i-\alpha}-\frac{2}{\theta}\\
&+\frac{\alpha\,{_2}F_1(-\theta,\alpha-\theta,1+\alpha-\theta;-1)+ \alpha\,{_2}F_1(-\theta,\alpha-\theta,1+\alpha-\theta;1)}{\theta(\alpha-\theta)},\\
\textsf{C}_{\theta}^{i}:=&\frac{\theta(\theta-1)\cdots(\theta-i+1)}{i!} \end{split}
\end{equation*} and $\,{_2}F_1(a,b,c;z)$ is the Gauss hypergeometric function which can be calculated by the formula as follows
\begin{equation*}
\begin{split}\,{_2}F_1(a,b,c;z):&=\sum_{n=0}^\infty \frac{(a)_n(b)_n z^n}{(c)_{n}n!},\\
(r)_0&=1,\,\, (r)_{n}=r(r+1)\cdots(r+n-1), \,\, n\ge 1.\end{split}
\end{equation*}

Since \begin{align*}{_2}F_1(-\theta,&\alpha-\theta,1+\alpha-\theta;-1)+{_2}F_1(-\theta,\alpha-\theta,1+\alpha-\theta;1)\\
&=\sum_{n=0}^\infty
\frac{(-\theta)_n(\alpha-\theta)_n}{(1+\alpha-\theta)_n}\frac{(-1)^n}{n!}+\sum_{n=0}^\infty
\frac{(-\theta)_n(\alpha-\theta)_n}{(1+\alpha-\theta)_n}\frac{1^n}{n!}\\
&=2\sum_{i=0}^\infty
\frac{(-\theta)_{2i}(\alpha-\theta)_{2i}}{(1+\alpha-\theta)_{2i}}\frac{1}{(2i)!}\\
&=2+2\sum_{i=1}^\infty \frac{\theta(\theta-1)\cdots (\theta-2i+1)
(\alpha-\theta)}{(2i+\alpha-\theta)}\frac{1}{(2i)!},\end{align*}
we arrive at
\begin{equation}\label{prooflemma2-2-3}
\begin{split}
  E(\alpha,\theta) =&\frac{2}{\alpha-\theta}+{2\alpha}{}\sum_{i=1}^\infty
\frac{(\theta-1)\cdots
(\theta-2i+1)}{(2i)!}\bigg(\frac{1}{2i-\alpha}+\frac{1}{2i+\alpha-\theta}\bigg).\end{split}\end{equation}

It is easy to see that the series appearing in \eqref{prooflemma2-2-3} is absolutely convergent, and so
\begin{equation}\label{prooflemma2-2} \begin{split}\lim_{\theta\longrightarrow0}E(\alpha,\theta)&=\frac{2}{\alpha}+2\alpha\sum_{i=1}^\infty
\frac{(-1)\cdots
(-2i+1)}{(2i)!}\bigg(\frac{1}{2i-\alpha}+\frac{1}{2i+\alpha}\bigg)\\
&=\frac{2}{\alpha}-\sum_{i=1}^{\infty}\frac{4\alpha}{4i^{2}-\alpha^{2}}\\
&=\pi\cot (\pi\alpha/2),\end{split}\end{equation}
where in the last equality we have used the identity (see \cite[Formula 1.421.3, p.\ 44]{GR})
$$ {\cot }(\pi x)=\frac{1}{\pi x}+\frac{2x}{\pi}\sum_{i=1}^\infty \frac{1}{x^2-i^2},\quad x\in \R.$$
Since for $\alpha>1$, $\pi\cot (\pi\alpha/2)<0$, the desired assertion \eqref{lemma2-1} for $|x|$ large enough follows from \eqref{prooflemma2-1}-\eqref{prooflemma2-2}.
This completes the proof.
\end{proof}

We are now in the position to prove Theorem \ref{thm1}(i).

\begin{proof}[{\bf Proof of Theorem \ref{thm1}(i)}] Since $C_c^\infty(\R)$ is the core of $\FF^\mu$, it is enough to prove the desired inequality \eqref{thm1-1} for any $f\in C_c^\infty(\R)$.
Since $V\geq 1$, we have by \eqref{con1} and \eqref{lemma2-1} that
there exists a constant $r_0>0$ such that
$$\I_{B(0,r)^c}\le \frac{1}{C_3\Phi(r)}\frac{-\LL^{(\alpha)}_{>1}V}{V}+\frac{C_4}{C_3\Phi(r)}\I_{B(0,r_0)},\quad r\ge r_0,$$where $$C_3=-\frac{\theta \pi C_{1,\alpha} \cot (\alpha\pi/2)}{2\alpha}, \quad C_4=\frac{2C_{1,\alpha}\sup_{|x|\le r_0}a(x)}{\alpha-\theta}.$$
Then, for any $f\in C_c^\infty(\R)$,
$$\mu(f^2\I_{B(0,r)^c})\le
\frac{1}{C_3\Phi(r)}\mu\left(f^2\frac{-\LL^{(\alpha)}_{>1}V}{V}\right)+\frac{C_4}{C_3\Phi(r)}\mu(f^2\I_{B(0,r_0)}),\quad
r\ge r_0.$$ Note that for any $x$, $y\in\R$,
\begin{equation*}
\begin{split}
\left(\frac{f^2(x)}{V(x)}-\frac{f^2(y)}{V(y)}\right)(V(x)-V(y))&=f^2(x)+f^2(y)-\left(\frac{V(y)}{V(x)}f^2(x)+\frac{V(x)}{V(y)}f^2(y)\right)\\
&\le f^2(x)+f^2(y)-2|f(x)||f(y)|\\
&\le (f(x)-f(y))^2,
\end{split}
\end{equation*}
which, together with Proposition \ref{pro1}, yields that
$$\mu\left(f^2\frac{-\LL^{(\alpha)}_{>1}V}{V}\right)\le
\wh \EE(f,f).$$ Therefore, for any $f\in C_c^\infty(\R)$,
\begin{equation}\label{proof1} \mu(f^2\I_{B(0,r)^c})\le
\frac{1}{C_3\Phi(r)}\wh \EE(f,f)+\frac{C_4}{C_3\Phi(r)}\mu(f^2\I_{B(0,r_0)}),\quad
r\ge r_0.
\end{equation}
On the other hand, by Lemma \ref{lemma1}, for any $f\in
C_c^\infty(\R)$ with $\mu(f)=0$, and $r\ge r_0$,
\begin{equation*}
\begin{split}
 \mu(f^2\I_{B(0,r)})
\le
&\frac{C_5K(r)^2 r^{\alpha}}{k(r)}\,   \EE (f,f)+\frac{1}{\mu(B(0,r))}\left(\int_{B(0,r)}f\,d\mu\right)^2\\
=&\frac{C_5K(r)^2 r^{\alpha}}{k(r)} \,  \EE (f,f)
+\frac{1}{\mu(B(0,r))}\left(\int_{B(0,r)^c}f\,d\mu\right)^2,
\end{split}
\end{equation*}
where $C_5=2^\alpha/ C_{1,\alpha}$, and in the equality above we have used the fact that $\int_{B(0,r)}f\,d\mu=-\int_{B(0,r)^c}f\,d\mu$.
The Cauchy-Schwarz inequality yields that
$$
\left(\int_{B(0,r)^c}f\,d\mu\right)^2\le
\mu(B(0,r)^c)\int_{B(0,r)^c}f^2\,d\mu .
$$
Hence, for any $f\in C_c^\infty(\R)$ with $\mu(f)=0$, and $r\ge
r_0$,
\begin{equation}\label{e:3.9}
\mu(f^2\I_{B(0,r)})\le\frac{C_5K(r)^2 r^{\alpha}}{k(r)}  \EE(f,f)
+\frac{\mu(B(0,r)^c)}{\mu(B(0,r))}\mu(f^2 \I_{B(0, r)^c}).
\end{equation}
Putting it into the right hand side of \eqref{proof1} and noting that $\wh \EE(f,f)\le \EE(f,f)$,
 we get that for any $f\in
C_c^\infty(\R)$ with $\mu(f)=0$, and $r\ge r_0$,
\begin{equation}\label{e:3.10--}
\mu(f^2\I_{B(0,r)^c})\le
\left( \frac{1}{C_3\Phi(r)}+\frac{C_4}{C_3\Phi(r)} \frac{C_5K(r)^2 r^{\alpha}}{k(r)} \right) \EE(f,f)+\frac{C_4}{C_3\Phi(r)}\frac{\mu(B(0,r)^c)}{\mu(B(0,r))}\mu(f^2).
\end{equation}
Summing \eqref{e:3.9} and \eqref{e:3.10--}, we have
$$\mu(f^2)\le \left( \frac{1}{C_3\Phi(r)}+\left(\frac{C_4}{C_3\Phi(r)} +1\right) \frac{C_5K(r)^2 r^{\alpha}}{k(r)} \right)  \EE(f,f)
+\left( \frac{C_4}{C_3\Phi(r)}+1\right) \frac{\mu(B(0,r)^c)}{\mu(B(0,r))} \mu(f^2).
$$
Under \eqref{con1}, we can choose $r_1\ge r_0$ such that
$$
\left( \frac{C_4}{C_3\Phi(r_1)}+1\right) \frac{\mu(B(0,r_1)^c)}{\mu(B(0,r_1))}
\le \frac{1}{2},
$$
and so we arrive at
$$
\mu(f^2)\le 2 \left( \frac{1}{C_3\Phi(r_1)}+\left(\frac{C_4}{C_3\Phi(r_1)} +1\right) \frac{C_5K(r_1)^2 r_1^{\alpha}}{k(r_1)} \right)  \EE(f,f).
$$
This gives the desired Poincar\'{e} inequality \eqref{thm1-1} with
$$C=2 \left( \frac{1}{C_3\Phi(r_1)}+\left(\frac{C_4}{C_3\Phi(r_1)} +1\right) \frac{C_5K(r_1)^2 r_1^{\alpha}}{k(r_1)} \right).$$
 The equivalence of the Poincar\'{e} inequality
and the corresponding bound of $\|P_t-\mu\|_{L^2(\R;\mu)}$ is well
known, see, e.g.,  \cite[Theorem 1.1.1]{WBook}.
\end{proof}

To prove Theorem \ref{thm1}(ii), we  need the following local
super Poincar\'{e} inequality for $\EE$.

\begin{lemma} \label{lemma3}
There exists a constant $C_6>0$ such that for any $s$, $r>0$ and
any $f\in C_c^\infty(\R)$,
\begin{eqnarray*}
\int_{B(0,r)}f(x)^2\,\mu(dx) &\le &
sC_{1, \alpha}\int_{B(0,r)\times B(0,r)}\frac{(f(y)-f(x))^2}{|y-x|^{1+\alpha}}\,dy \,\mu(dy) \\
 & &+ \frac{C_6K(r) }{k(r)^2} \big(1+(C_{1, \alpha} s)^{-1/\alpha} K(r)^{1/\alpha}\big)\bigg(\int_{B(0,r)}|f|(x)\,\mu(dx)\bigg)^2\\
& \le & s\EE(f,f)+ \frac{C_6K(r) }{k(r)^2} \big(1+(C_{1, \alpha} s)^{-1/\alpha} K(r)^{1/\alpha}\big)\mu(|f|)^2.
\end{eqnarray*}
\end{lemma}

\begin{proof} The second inequality is trivial, and so we only need
the consider the first inequality. The Sobolev inequality of dimension
$2/\alpha$ for fractional Laplacians in $\R$ holds uniformly on
balls, e.g.\ see \cite[Theorem 3.1]{CK2}.
 Thus according to  \cite[Corollary 3.3.4]{WBook},
there exists a constant $c_1>0$ such that
\begin{equation*}
\int_{B(0,r)} f^2(x)\,dx
\le s\int_{B(0,r)\times B(0,r)}
\frac{(f(y)-f(x))^2}{|y-x|^{d+\alpha}}\,d y\,d x +
c_1\big(1+s^{-1/\alpha}\big)  \bigg(\int_{B(0,r)} |f(x)|\,dx\bigg)^2
\end{equation*}
 holds for all $f\in C_c^\infty(\R)$ and all $s,r>0.$
Therefore,
\begin{eqnarray*}
\int_{B(0,r)} f^2(x)\,\mu(dx)
 &\le & K(r)\int_{B(0,r)} f^2(x)\,dx\\
& \le& sK(r)\int_{B(0,r)\times B(0,r)}
\frac{(f(y)-f(x))^2}{|y-x|^{d+\alpha}}\,d y\,d x\\
 & &
+ c_1\big(1+s^{-1/\alpha}\big)K(r)k(r)^{-2}  \bigg(\int_{B(0,r)} |f(x)|\,\mu(dx)\bigg)^2 .
\end{eqnarray*}
This implies the first desired assertion  by replacing $s$ with $C_{1,\alpha}sK(r)^{-1}$ in the above inequality.
\end{proof}

\begin{proof}[{\bf Proof of Theorem \ref{thm1}(ii)}]
 As in the proof of Theorem \ref{thm1}(i), it suffices to consider $f\in C_c^\infty(\R)$. First, by Lemma \ref{lemma3}, there exists a constant $C_7>0$ that
depends on $C_6$, $C_{1, \alpha}$ and $K(0+)$
such that
$$\mu(f^2\I_{B(0,r)})\le s\EE (f,f)+C_7 K_0(r) (1+s^{-1/\alpha})\mu(|f|)^2,\quad r,s>0, f\in C_c^\infty(\R)
$$
where
$$ K_0 (r) =  \frac{  K(r)^{1+1/\alpha}}{k(r)^2}.
$$
Combining this with \eqref{proof1} and noting that $\wh \EE(f,f)\le \EE(f,f)$, there exists a
constant $c_0>0$ such that for any $r\ge r_0$, $s>0$ and  $f\in C_c^\infty(\R)$,
$$
\mu(f^2)\le\left(\frac{c_0}{\Phi(r)}
+s\Big(1+\frac{c_0}{\Phi(r)}\Big)\right)\EE(f,f)+
 \left(1+\frac{c_0}{\Phi(r)}\right) K_0(r) (1+s^{-1/\alpha}) \mu(|f|)^2 .
$$
By \eqref{con2}, letting $s_0={c_0}/{\Phi(r_0)}$ and taking
$r=\Phi^{-1}(c_0/s)$, we have for $s\in (0,s_0]$ and
$f\in C^\infty_c(\R)$,
$$
\mu(f^2)\le (2s+s^2) \EE(f,f)+ (1+s )(1+s^{-1/\alpha}) K_0 \circ \Phi^{-1}(c_0/s) \mu(|f|)^2 .
$$
In particular, for  $s\in (0,s_0\wedge1]$ and
$f\in C^\infty_c(\R)$,
$$
\mu(f^2)\le 3s \EE(f,f)+2(1+s^{-1/\alpha})K_0 \circ \Phi^{-1}(c_0/s)\mu(|f|)^2.
$$
Replacing $s$ by $s/3$, we get for $s\in (0,3(s_0\wedge1)]$ and
$f\in C^\infty_c(\R)$,
$$
\mu(f^2)\le s \EE(f,f)+2 \cdot 3^{1/\alpha}(1+s^{-1/\alpha})
K_0 \circ \Phi^{-1}(3c_0/s)
\mu(|f|)^2.
$$
This proves the super Poincar\'{e} inequality \eqref{thm2-1}
with $\beta (r)=  2 \cdot 3^{1/\alpha}(1+r^{-1/\alpha}) K_0 \circ \Phi^{-1}(3c_0/r) $
for $r\in (0,3(s_0\wedge1))$
and $\beta (r)= \beta(3(s_0\wedge1))$ when $r\geq 3(s_0\wedge1)$. Combining both estimates for $\beta$ yields the desired assertion \eqref{e:1.9}.

Define
$ \Psi(t)=\int_t^\infty \frac{\beta^{-1}(r)}{r}dr$
for $t>\inf \beta (r)$.
Thus by \cite[Theorem 3.3.14]{WBook},
$$\|P_t\|_{L^1(\R;\mu)\to L^\infty(\R;\mu)}\le
2\Psi^{-1}(t)  \quad  \hbox{for } t>0,
$$
where
$$\Psi^{-1}(t)=\inf \{r\ge \inf \beta:
\Psi(r)\le t\}.
$$
In particular, we conclude that $P_{t}$ has a bounded
kernel $p(t, x, y)$ with respect to $\mu$ for every $t>0$.
Moreover, since every point is visited by the time-changed process
$Y$ and thus having positive capacity,
we have by \cite[Theorem 3.1]{BBCK} that the symmetric kernel $p(t, x, y)$
can be chosen in such a way that for every fixed $y\in \R$ and $t>0$,
$x\to p(t, x, y)$ is bounded and quasi-continuous on $\R$.
Thus for every $f\in L^2(\R; \mu)$ and $t>1$, we have by the Cauchy-Schwarz inequality and
\eqref{e:1.6} that
\begin{eqnarray*}
\sup_{x\in \R} | P_t f(x)-\mu (f)|
&=&  \sup_{x\in \R} | P_{1} (P_{t-1} f-\mu (f))(x)|
\\
&\le &\sup_{x\in \R} p_{2} (x, x)^{1/2} \,  \| P_{t-1} f -\mu (f)\|_{L^2(\R; \mu)}
 \\
&\leq & \ c e^{-t/C} \| f\|_{L^2(\R; \mu)}.
\end{eqnarray*}
This establishes \eqref{e:1.10}. The last assertion is a direct consequence of \eqref{e:1.10}, see, e.g.,\ \cite[Theorem 8.5]{Chen}.
\end{proof}

Next, we  prove Theorem \ref{thm1}(iii).

\begin{proof}[{\bf Proof of Theorem \ref{thm1}(iii)}] Let $a_0(x)=(1+|x|)^\alpha$ and $\mu_0(dx)=\frac{a_0(x)^{-1}\,dx}
{\int a_0(x)^{-1}\,dx}$. According to Theorem \ref{thm1}(i), the
following Poincar\'{e} inequality $$ \mu_0(f^2)\le
C\EE(f,f),\quad f\in \FF^\mu, \mu_0(f)=0$$ holds for
some constant $C>0$. Therefore, for any $s>0$ and $f\in
\FF^\mu$,
\begin{equation*}\begin{split}
&\int_{B(0,s)}\left(f(x)-\frac{1}{\mu(B(0,s))}\int_{B(0,s)}f(x)\,\mu(dx)\right)^2\,\mu(dx)\\
&=\inf_{a\in\R}\int_{B(0,s)}(f(x)-a)^2\,\mu(dx)\\
&\le \int_{B(0,s)}(f(x)-\mu_0(f))^2\,\mu(dx)\\
&\le c_1\left(\sup_{|x|\le
s}\frac{a(x)^{-1}}{a_0(x)^{-1}}\right)\int_{B(0,s)}(f(x)-\mu_0(f))^2\,\mu_0(dx)\\
&\le \frac{c_2}{\Phi_0(s)} \EE(f,f).
\end{split}\end{equation*}
The desired weak Poincar\'{e} inequality
\eqref{thm3-1} now follows from \eqref{con3} and \cite[Theorem 4.3.1]{WBook}.

By \cite[Theorem 4.1.3]{WBook} and its proof,
\eqref{thm3-1} implies that
$$ \| P_t f -\mu (f) \|_{L^2(\R; \mu)}^2 \leq
\eta (t) \left( \| f-\mu (f)\|_{L^2(\R; \mu)}^2+
\| f -\mu (f)\|_\infty^2\right)
\leq  2 \eta (t) \left( \| f \|_{L^2(\R; \mu)}^2+
\| f  \|_\infty^2\right)
$$
for every $ f\in \FF^\mu$, where
$$ \eta (t) =  \inf\{r>0: -  \alpha (r) \log r \leq 2 t\}.
$$
This completes the proof of Theorem \ref{thm1}(iii).
\end{proof}

\medskip

To prove Examples \ref{exa1} and \ref{exa2}, for $n\geq 1$, let
$g_n\in C^\infty(\R)$ be such that
\begin{equation}\label{e:3.10}
g_n(x)\begin{cases}=0, & |x|\le n;\\
  \in[0,1], & n\le |x|\le 2n;\\
  =1,&|x|\ge 2n,\end{cases}
\end{equation}
and $|g_n'(x)|\le 2/n$ for all $x\in\R$.
By \eqref{e:1.2} and \eqref{e:1.3}, $g_n\in \FF_e \cap L^2 (\R; \mu)=\FF^\mu$.
 Moreover, there exists a
constant $c>0$ independent of $n$ such that for all $n\ge1$,
\begin{align}\label{ref}
\EE(g_n,g_n)&=\frac{C_{1,\alpha}}{2} \int_{\{|x|\le 3n\}}
\int_{\{|y|\le 4n\}}\frac{(g_n(x)-g_n(y))^2}{|x-y|^{1+\alpha}}\,dy\,dx \nonumber\\
& \quad+ \frac{C_{1,\alpha}}{2} \int_{\{|x|\le 3n\}}\int_{\{|y|> 4n\}}
\frac{(g_n(x)-g_n(y))^2}{|x-y|^{1+\alpha}}\,dy\,dx \nonumber \\
&\quad
+\frac{C_{1,\alpha}}{2} \int_{\{|x|\ge 3n\}}   \int_{\{|y|\leq 2n\}}
\frac{(g_n(x)-g_n(y))^2}{|x-y|^{1+\alpha}}\,dy\,dx  \nonumber \\
&\le \frac{2C_{1,\alpha}}{n^2}\int_{\{|x|\le 3n\}} \left( \int_{\{|x-y|\le 7n\}}\frac{|x-y|^2}{|x-y|^{1+\alpha}}\,dy\right) dx\nonumber \\
&\quad+ \frac{C_{1,\alpha}}{2}
\int_{\{|x|\le 3n\}} \left( \int_{\{|x-y|\ge  n\}}\frac{1}{|x-y|^{1+\alpha}}\,dy\right) dx \nonumber \\
&\quad +\frac{C_{1,\alpha}}{2} \int_{\{|y|\leq 2n\}} \left( \int_{\{|x-y|\geq n\}}
\frac{1}{|x-y|^{1+\alpha}}\,dx \right) dy  \nonumber \\
&\le \frac{2C_{1,\alpha}}{n^2} \int_{\{|x|\le 3n\}} \left( \int_{\{|z|\le 7n\}}|z|^{1-\alpha}\,dz\right) dx\nonumber \\
&\quad+ \frac{C_{1,\alpha}}{2} \int_{\{|x|\le 3n\}}
\left( \int_{\{|z|\ge n\}} \frac{1}{|z|^{1+\alpha}}\,dz\right) dx \nonumber\\
&\quad+\frac{C_{1,\alpha}}{2} \int_{\{|y|\leq 2n\}} \left( \int_{\{|z|\geq n\}}
\frac1{|z|^{1+\alpha}}\,dz\right) dy  \nonumber \\
&= c n^{1-\alpha}.
\end{align}

\medskip

\begin{proof}[{\bf Proof of Example \ref{exa1}}] (i) If $\gamma\ge \alpha$, we have $\Phi(0)=C_\gamma$.
So the Poincar\'{e} inequality \eqref{thm1-1} follows from Theorem
\ref{thm1}(i).

To disprove the Poincar\'{e} inequality \eqref{thm1-1} for
$\gamma\in (1,\alpha)$, let
$g_n$ be the function defined by \eqref{e:3.10}. It is clear that
there are constants $c_1, c_2>0$ so that
$$
\mu(g_n^2)\ge \frac{c_1}{n^{\gamma-1}} \quad \hbox{and} \quad
\mu(g_n)^2\le \frac{c_2}{n^{2(\gamma-1)}} \qquad \hbox{for } n\geq 1.
$$
 Combining these with \eqref{ref}, we have
$$
\lim_{n\to\infty}
\frac{\EE(g_n,g_n)}{\mu(g_n^2)-\mu(g_n)^2}\le \lim_{n\to\infty}
\frac{cn^{-\alpha+1}}{c_1n^{-\gamma+1}-c_2n^{-2\gamma+2}}=0.
$$
Therefore, for any constant $C>0$, the Poincar\'{e} inequality
\eqref{thm1-1} does not hold.

(ii) Let $\gamma>\alpha$. It is easy to see that $\Phi(r)=C_\gamma(1+r)^{\gamma-\alpha}$,
$$K(r)=C_\gamma^{-1}, \quad  k(r)=C_\gamma^{-1}(1+r)^{-\gamma} \quad \textrm{ and } \quad
K_0 (r)= C_\gamma^{1-1/\gamma}(1+r)^{2\gamma}.
$$
Then, $ \Phi^{-1}(r)= (r/C_\gamma)^{1/(\gamma-\alpha)}-1$ for $r>1$ and so
$$\beta (r)\leq c (1+r^{-(1/\alpha + 2\gamma/(\alpha-\gamma)})
\quad \hbox{for } r>0.
$$
Observe that $\beta^{-1}(r)/r$ is integrable near infinity. Thus \eqref{e:1.10} holds.
Moreover, by \cite[Theorem 3.3.15(2)]{WBook}, the corresponding bound \eqref{e:1.15}
on $\| P_t\|_{L^1(\R; \mu)\to L^\infty (\R; \mu)}$ holds.

We next show that when $\gamma\in(1,\alpha]$,  the super Poincar\'{e} inequality \eqref{thm2-1} fails. Indeed, if the inequality \eqref{thm2-1} holds
for some non-increasing positive function $\beta (r)$,
then, applying it to the function $g_n$ defined in \eqref{e:3.10}, we get
$$\frac{c_3}{n^{\gamma-1}}\le \frac{cr}{n^{\alpha-1}}+ \frac{c_4\beta(r)}{n^{2(\gamma-1)}},\quad r>0, n\ge1
$$
 for some constants $c, c_3, c_4>0$. Since $\gamma\in(1,\alpha]$,
 one has
$$
c_3\le \lim_{n\to\infty} \Big(\frac{cr}{n^{\alpha-\gamma}}+ \frac{c_4\beta(r)}{n^{\gamma-1}}\Big)\le cr \quad \hbox{for every } r>0,
$$
which is impossible.

(iii) Let $\gamma\in(1,\alpha)$. Then
$\Phi_0(r)=C_\gamma(1+r)^{\gamma-\alpha}$, and so the desired weak
Poincar\'{e} inequality follows from Theorem \ref{thm1}(iii).
According to \cite[Theorem 4.1.5(2)]{WBook}, we have the claimed
bound for $\|P_t-\mu\|_{L^\infty(\R;\mu)\to L^2(\R;\mu)}$. On the
other hand, for  function $g_n$ given by \eqref{e:3.10}, we have
$\|g_n\|_\infty=1$, $\mu(g_n^2)-\mu(g_n)^2\ge c_1n^{-(\gamma-1)}$
for some constant $c_1>0$. Hence, according to \eqref{thm3-1} and
\eqref{ref}, $$c_1n^{-(\gamma-1)}\le c\alpha(r)
n^{-(\alpha-1)}+2r,\quad r>0.$$ Taking
$r=r_n:=\frac{c_1}{4n^{\gamma-1}}$ which goes to zero as $n\to \infty$,
we get that
$$\liminf_{n\to\infty}
r_n^{(\alpha-\gamma)/(\gamma-1)}\alpha(r_n)>0.$$ Thus, the
weak Poincar\'{e} inequality \eqref{thm3-1} does not hold if
$\lim_{r\to0}r^{(\alpha-\gamma)/(\gamma-1)}\alpha(r)=0.$
\end{proof}

\begin{proof}[{\bf Proof of Example \ref{exa2}}]
(i) If $\gamma\ge 0$, we have $\Phi(0)=C_{\alpha,\gamma}$ and so
the Poincar\'{e} inequality \eqref{thm1-1} follows from Theorem
\ref{thm1}(i).

To prove that the Poincar\'{e} inequality \eqref{thm1-1} does not hold for
$\gamma<0$, we take   function
$g_n$ defined by \eqref{e:3.10}. It is clear that $$\mu(g_n^2)\ge
\frac{c_1}{n^{\alpha-1}\log ^\gamma n},\quad \mu(g_n)^2\le
\frac{c_2}{n^{2(\alpha-1)}\log ^{2\gamma} n}$$ hold for $n$ large enough and some constants $c_1$,
$c_2>0$ (independent of $n$). Combining these with \eqref{ref}, we arrive at
$$\lim_{n\to\infty}
\frac{\EE(g_n,g_n)}{\mu(g_n^2)-\mu(g_n)^2}\le \lim_{n\to\infty}
\frac{cn^{-\alpha+1}}{c_1n^{-\alpha+1}\log^{-\gamma} n-c_2n^{-2\alpha+2}\log^{-2\gamma} n}=0$$ provided  $\gamma<0$.
This implies that the Poincar\'{e} inequality
\eqref{thm1-1} does not hold for any constant $C>0$.

(ii) Let $\gamma>0$. It is easy to see that $\Phi(r)=C_{\alpha,\gamma}\log^{\gamma}(e+r)$,
$$
K(r)=C_{\alpha,\gamma}^{-1}, \quad  k(r)=C_{\alpha,\gamma}^{-1}(1+r)^{-\alpha}\log^{-\gamma}(e+r) \quad \textrm{ and } \quad
K_0(r)= C_{\alpha,\gamma}^{1-1/\alpha}(1+r)^{2\alpha} \log^{2\gamma} (e+r).
$$
 Then $ \Phi^{-1}(r)=\exp\big((r/C_{\alpha,\gamma})^{1/\gamma}\big)-e$ and so the super Poincar\'e inequality \eqref{thm2-1} holds with
$$\exp\Big( c_1(1+r^{-1/\gamma})\Big) \leq
\beta(r)\le \exp\Big( c_2(1+r^{-1/\gamma})\Big),\quad r>0$$
for some constant $c_1, c_2>0$.
When $\gamma>1$, we get \eqref{e:1.16} from  \cite[Theorem 3.3.15(1)]{WBook}.
Moreover, $\beta^{-1}(r)/r$ is integrable at infinity if and only if $\gamma>1$.
In this case, the uniform strong ergodicity \eqref{e:1.10} holds.

Next, we prove that if $\gamma\le 0$, then for any $\beta:(0,\infty)\to(0,\infty)$ the super Poincar\'{e} inequality \eqref{thm2-1} does not hold. Indeed, if the inequality \eqref{thm2-1} holds, then, applying it to the function $g_n$ of \eqref{e:3.10},
$$\frac{c_3}{n^{\alpha-1}\log^\gamma n}\le \frac{cr}{n^{\alpha-1}}+ \frac{c_4\beta(r)}{n^{2(\alpha-1)}\log^{2\gamma}n},\quad r>0, n\ge1$$ holds for some constants $c, c_3, c_4>0$. Since $\gamma\le0$, we get that
$$c_3\le \lim_{n\to\infty} \Big(\frac{cr}{\log^{-\gamma}n}+ \frac{c_4\beta(r)}{n^{\alpha-1}
\log^{\gamma}n}\Big)\le cr, \quad r>0.$$ Letting $r\to0$ we get that $c_3\le 0$, which is impossible.

Consider the  function $g_n$ defined by  \eqref{e:3.10}. It is easy
to see that
$$\mu(g_n^2)\ge \frac{c_5}{n^{\alpha-1}\log^\gamma(e+n)},\quad \mu(g_n)^2\le \frac{c_6}{n^{2(\alpha-1)}\log^{2\gamma}(e+n)}$$ hold for all $n\ge1$ and some constants $c_5, c_6>0$. Combining these with \eqref{ref} and \eqref{thm2-1}, we have
$$\frac{c_5}{\log^\gamma(e+n)}\le cr+ \frac{c_6\beta(r)}{n^{\alpha-1}\log^{2\gamma}(e+n)},\quad r>0.$$Taking $r=r_n:=\frac{c_5}{2c\log^\gamma(e+n)}$, we derive
$$\beta(r_n)\ge \frac{c_5}{2c_6}n^{\alpha-1}\log^\gamma(e+n),\quad n\ge1.$$ Therefore, $$\liminf_{n\to\infty}r_n^{1/\gamma}\log\beta(r_n)\ge \alpha-1>0.$$

According to \cite[Corollary 3.3.4(1)]{WBook}, the super
Poincar\'{e} inequality with $\beta(r)=\exp(c(1+r^{-1}))$ for some
$c>0$ is equivalent to the following defective log-Sobolev
inequality
\begin{equation}\label{delogine}\mu(f^2\log f^2)\le C_1\EE(f,f)+C_2 \quad  \hbox{for } f\in \FF^\mu \hbox{ with } \mu(f^2)=1,
\end{equation} where $C_1$ and $C_2$ are two positive constants. Since the
symmetric Dirichlet form $(\EE, \FF^\mu)$ is conservative and irreducible,
it follows from \cite[Corollary 1.3]{W13} that
the defective log-Sobolev inequality \eqref{delogine} is also
equivalent to log-Sobolev  inequality \eqref{logine}. It is well known
that  the log-Sobolev inequality implies the entropy of the semigroup $P_t$
decays exponentially  \eqref{e:1.17},
 see, e.g., \cite[Corollary 1.1]{Ch}.

(iii) Let $\gamma<0$. Then there exist two constants $c_7$ and
$c_8>0$ such that for all $r>0$, $$\Phi_0(r)=
c_7\log^{\gamma}(1+r),\quad \mu(B(0,r)^c)=
c_8r^{-(\alpha-1)}\log^{-\gamma}(1+r).$$ Then the desired weak
Poincar\'{e} inequality follows from Theorem \ref{thm1}(iii), and so
the corresponding bound for $\|P_t-\mu\|_{L^\infty(\R;\mu)\to
L^2(\R;\mu)}$ follows from \cite[Theorem 4.1.5(1)]{WBook}. On the
other hand,   for function $g_n$ defined in \eqref{e:3.10},
we have $\|g_n\|_\infty=1$ and $$\mu(g_n^2)-\mu(g_n)^2\ge
c_9n^{-(\alpha-1)}\log^{-\gamma}(1+n)$$ for some constant
$c_9>0$. Hence, according to \eqref{thm3-1} and \eqref{ref},
$$c_9n^{-(\alpha-1)}\log^{-\gamma}(1+n)\le c\alpha(r)
n^{-(\alpha-1)}+2r,\quad r>0.$$   Taking
$r=r_n:=\frac{c_9n^{-(\alpha-1)}\log^{-\gamma}(1+n)}{4}$ which
goes to zero as $n\to \infty$, we get that
$$\liminf_{n\to\infty}
\log^\gamma(1+r_n^{-1})\alpha(r_n)=\alpha-1>0.$$ Thus, the
weak Poincar\'{e} inequality \eqref{thm3-1} fails if
$\lim_{r\to0}\log^\gamma(1+r^{-1})\alpha(r)=0.$
\end{proof}

We next prove Theorem \ref{thm2}.

\begin{proof}[{\bf Proof of Theorem \ref{thm2}}] We only need to consider $f\in C_c^\infty(\R)$.
First, according to Lemma \ref{lemma3}, there exists a constant $c_1>0$ such that
$$\mu(f^2\I_{B(0,r)})\le s \EE (f,f)+c_1K_0(r)(1+s^{-1}) \mu(|f|)^2,\quad r,s>0, f\in C_c^\infty(\R)
$$
On the other hand, for any $1< \delta< \beta$, let $V$ be the function defined in Lemma \ref{lemma2}.
Then, by Lemma \ref{lemma2} and \eqref{thm2-1-1}, we know that
$$
\LL^{(\delta)}_{>1}V(x)\le -c_2\Psi_\beta(|x|)(1+|x|)^{\beta-\delta}
V(x)\I_{B(0,r_0)^c}(x)+c_3\I_{B(0,r_0)}(x)
$$
for some positive constants $c_2$, $c_3$ and $r_0$ depending on $\delta$.
Combining this with the argument of Theorem \ref{thm1}(ii), and noting that
$$
\int_{\R} f(x) (-\LL^{(\delta)}_{>1}V(x)) \mu (dx) \leq
 \EE (f,f),
 $$
 we conclude that there exists a
constant $c_4>0$ such that for any $r\ge r_0$, $s>0$ and $f\in C_c^\infty(\R)$,
\begin{equation*}\begin{split}\mu(f^2)\le
&\left(\frac{c_4}{\Psi_\beta(r)(1+r)^{\beta-\delta}}+s \Big(1+\frac{c_4}{\Psi_\beta(r)(1+r)^{\beta-\delta}}\Big)\right)\EE(f,f)\\
&+
c_1K_0(r)(1+s^{-1}) \left(1+\frac{c_4}{\Psi_\beta(r)(1+r)^{\beta-\delta}}\right)\mu(|f|)^2\\
=:&\left(\frac{c_4}{\Psi_{\beta,\delta}(r)}+s\Big(1+\frac{c_4}{\Psi_{\beta,\delta}(r)}\Big)\right)\EE(f,f)+
c_1K_0(r)(1+s^{-1})
\left(1+\frac{c_4}{\Psi_{\beta,\delta}(r)}\right)\mu(|f|)^2.
\end{split}\end{equation*} Taking
$r=\Phi_{\beta,\delta}^{-1}(c_4/s)$, we get for $s\in (0,s_0)$ and
$f\in C^\infty_c(\R)$ that
$$
\mu(f^2)\le (2s+s^2) \EE(f,f)+c_1 (1+s) (1+s^{-1})
K_0\circ\left(\Phi_{\beta,\delta}^{-1}(c_4/s)\right) \mu(|f|)^2 ,
$$
 where
$s_0=\frac{c_4}{\Psi_{\beta,\delta}(r_0)}$.
Replacing $s$ by $s/3$, we have for $ s\in (0,3(s_0\wedge1))$ and
$f\in C^\infty_c(\R)$,
$$\mu(f^2)\le s \EE(f,f)+ c_2  (1+s^{-1})  K_0\circ \left(\Phi_{\beta,\delta}^{-1}(3c_4/s)\right) \mu(|f|)^2  .
$$
 The required assertion follows from the
inequality above and by taking $\beta(s)=\beta(3(s_0\wedge1))$ for
all $s\ge 3(s_0\wedge1)$.
\end{proof}

The proof of Example \ref{exa3} is similar to that of Example \ref{exa1}(ii),
we omit the details here.
We now present the proof of Theorem \ref{pro100}.

\begin{proof}[{\bf Proof of Theorem \ref{pro100}}] (i) Since the
coefficient $a$ in the SDE \eqref{sde1} is locally
$1/\alpha$-H\"older continuous, it follows from \cite[Theorem
1.1]{B} or \cite[Theorem 1]{Ku} that \eqref{sde1} has a unique
strong solution $(Z_t)_{t\ge0}$ up to the explosion time
$$\tau=\inf\{t>0: Z_t\notin \R\}.$$ According to Lemma \ref{lemma2} and \cite[Theorem 2.1(1)]{MT2},
 $\tau=\infty$. This is, the SDE \eqref{sde1} has a
unique strong solution.

Let $(P_t)_{t\ge0}$ be the semigroup of the process $(Z_t)_{t\ge0}$.
It follows from \eqref{e:1.13} that $Z$ is a time-change of the
symmetric $\alpha$-stable process, and  $\mu$ is its symmetrizing
and unique invariant probability measure (see Proposition
\ref{P:1.1}).

In view of \eqref{e:1.13}, \eqref{pro1000} and Theorem
\ref{thm1}(i),   the semigroup $(P_t)_{t\ge0}$ is $L^2(\R;\mu)$
exponentially ergodic, which is equivalent to the desired assertion
for the exponential ergodicity of $(P_t)_{t\ge0}$ in the total
variation norm, see, e.g., \cite[Theorem 8.8]{Chen}.

(ii) Since $\alpha\in(1,2)$,
symmetric $\alpha$-stable process $X$ on $\R$ is pointwise recurrent,
so is $Z$.
This in particular implies that
the process $Z$ is Lebesgue
irreducible, that is, for any $x\in\R$ and any Borel
set $B$ with positive Lebesuge measure, $\Pp_x(\sigma_B<\infty)>0$, where $\sigma_B=\inf\{t\ge0:X_t\in B\}$.
Let $B$ be any Borel set in $\R$ with $\mu (B)=0$.
Denote by $(P_t)_{t\geq 0}$ the transition semigroup of $Y$.
Then for each $t>0$, $\mu (P_t \I_B)= \mu (B)=0$ and so $P_t  \I_B =0$
$\mu$-a.e. on $\R$. Since $P_t \I_B$ is $\EE$-quasi-continuous, it follows
that $P_t \I_B (x)=0$ for every $x\in \R$. This shows that for every
$x\in \R$ and $t>0$, $P_t (x, dy)$ is absolutely continuous with respect to
$\mu$. We claim that if $\mu (B)>0$, then $P_t (x, B)>0$ for every $x\in \R$
and $t>0$. Suppose there are $x_0\in \R$ and $t_0>0$ so that $P_{t_0} (x_0, B)=0$.
Then $P_{t_0/2} (P_{t_0/2}\I_B)(x_0)=0$ and so $P_{t_0/2}\I_B=0$ $\mu$-a.e. on $\R$.
The latter would imply that $\mu (B)=\mu (P_{t_0/2}\I_B)=0$, which is absurd.
This proves the claim, which in particular implies that $Y$ is
aperiodic.

On the other hand, let $\varphi\in C^2(\R)$ be a nonnegative function such that $\varphi(x)=|x|$
for $|x|\ge1$, and $\varphi(x)\le |x|$ for $|x|\le 1$. Under \eqref{pro2000},
define the function $V(x)=2-(1+\varphi(x))^{-\theta}$ for some constant
$\theta\in (0,\alpha(\gamma-1))$ to be determined later. It is easy to see
that $\LL V$ exists pointwise as a locally bounded function. Furthermore,
according to the proof of \cite[Theorem 1.1(ii)]{Sa} (also see \cite[(3.24)-(3.28)]{Sa} for more details),
we know that $$\limsup_{|x|\to\infty}(1+|x|)^{\alpha+\theta}\LL V(x)= \frac{\theta C_{1,\alpha} }{\alpha}E(\alpha,-\theta),$$ where $E(\alpha,-\theta)$ is defined by \eqref{prooflemma2-2-3}. By \eqref{prooflemma2-2}, one can choose $\theta\in (0,\gamma-\alpha)$ small enough such that
$$\limsup_{|x|\to\infty}(1+|x|)^{\alpha+\theta}\LL V(x)\le
\frac{\theta C_{1,\alpha}\pi }{2\alpha} \cot (\pi\alpha/2).$$
Combining \eqref{e:1.13}, \eqref{pro2000} with all the conclusions
above, we get that there exist $c_1$, $c_2$ and $r_0>0$ such that
\begin{equation}\label{drift}\LL V\le -c_1
V+c_2\I_{B(0,r_0)}.\end{equation}

As mentioned above, the process $Z$ is Lebesgue irreducible and
aperiodic. Therefore, according to \eqref{drift} and \cite[Theorem
5.2(c)]{DMT},   the process $Z$ is uniformly strongly
ergodic.
\end{proof}

\section{Appendix: ergodicity of time changed Brownian motions in dimension one}
Let $B$ be a Brownian motion on $\R$, and $a$ be a positive and locally bounded
measurable function on $\R$ such that $\mu(dx):=\frac{1}{a(x)}dx$ is
a probability measure (that is, $\int_{\R} a(x)^{-1}dx=1$). For any
$t>0$, define $$A_t=\int_0^t\frac{1}{a(B_s)}ds,\quad \tau_t=\inf\{s>0: A_s>t\},$$
and set $Y_t=B_{\tau_t}$. Following Section \ref{section1}, we know
that the time changed Brownian motion $Y$ is a recurrent $\mu$-symmetric diffusion
process on $\R$, and $\mu$ is the unique invariant probability measure
of $Y$. The Dirichlet form $(\EE, \FF^{\mu})$ of $Y$ is
$$\EE(f,f)=\frac{1}{2}\int_{\R} f'(x)^2\,dx,$$ where $\FF^{\mu}$ is defined
by \eqref{e:1.3}. In this appendix, we   present results on various
ergodic properties of $Y$ in terms of Poincar\'{e} type
inequalities. For readers's convenience, we  also present the sketch of the proof below.

\begin{theorem} \label{thm4}
For any $r>0$, set
$$\Phi(r):=\inf_{|x|\ge r}\frac{a(x)}{(1+|x|)^2},\quad
\Phi_0(r):=\inf_{|x|\le r} \frac{a(x)}{(1+|x|)^2}$$ and
$$
K(r):=\sup_{|x|\leqslant r}a(x)^{-1},\quad k(r):=\inf_{|x|\leqslant
r}a(x)^{-1}, \quad \quad K_0(r):=
\frac{K(r)^{3/2}}{k(r)^2}.
$$
Then the following holds.

\begin{itemize}
\item[\rm (i)] The
Poincar\'{e} inequality \eqref{thm1-1} holds for some constant $C>0$
if and only if
\begin{equation}\label{thm4-1}\limsup_{|x|\to\infty}\frac{1}{|x|}\int_{\{|s|\ge |x|\}} \frac{ds}{a(s)}<\infty.\end{equation}
In particular, if $\lim_{r\to\infty}\Phi(r)>0$, then \eqref{thm4-1}
is satisfied.

\item[\rm (ii)] The super
Poincar\'{e} inequality \eqref{thm2-1} holds if and only if
\begin{equation}\label{thm4-2}\limsup_{|x|\to\infty}\frac{1}{|x|}\int_{\{|s|\ge |x|\}}
\frac{ds}{a(s)}=0.
\end{equation} Moveover, in this case \eqref{thm2-1} holds
with
\begin{equation}\label{thm4-3}
\beta(r)=C_1\left(1+r^{-1/2} K_0(  \Phi^{-1}(C_2/r))\right)
\end{equation}
with some constants $C_1$ and $C_2>0$.
In particular, if $ \lim_{r\to\infty}\Phi(r)=\infty,$ then
\eqref{thm4-2} is satisfied.

\item[\rm (iii)]If $
\lim_{r\to\infty}\Phi_0(r)=0,$ then the weak
Poincar\'{e} inequality $$\mu(f^2)\le \alpha(r) \EE
(f,f)+r\|f\|_\infty^2 \quad \hbox{for every } r>0, f\in \FF^\mu,
\mu(f)=0$$ holds with $$\alpha(r)=C\inf\Big\{\frac{1}{\Phi_0(s)}: \mu(B(0,s))\geq \frac{1}{1+r}\Big\}
$$
for some constant $C>0$ independent of $r$.
\end{itemize}
\end{theorem}

\begin{proof}[{\bf Sketch of  proof for Theorem \ref{thm4}}] The assertion (i) immediately
follows from criterion for the Poincar\'{e} inequality of
one-dimensional diffusion processes, see, e.g.,\ \cite[Table 5.1, p.\
100]{Chen}. Having (i) at hand, one can follow the proof of Theorem
\ref{thm1}(iii) to obtain the assertion (iii).

The first statement of assertion (ii) is a direct consequence of
criterion for the super Poincar\'{e} inequality of one-dimensional
diffusion processes, also see \cite[Table 5.1, p.\ 100]{Chen}. To
verify \eqref{thm4-3}, we first note that the classical Sobolev
inequality for Laplacian holds uniformly on balls. Then, by
\cite[Corollary 3.3.4(2)]{WBook}, there exists a constant $c_1>0$
such that for all $f\in C_c^\infty(\R)$ and $r,s>0$,
$$\int_{B(0,r)}f(x)^2dx\le s
\int_{B(0,r)}f'(x)^2dx+c_1(1+s^{-1/2})\Big(\int_{B(0,r)}|f|(x)dx\Big)^2.$$
Therefore, we can conclude that there is  a constant $c_2>0$ such
that for all $f\in C_c^\infty(\R)$ and $r,s>0$,
$$\int_{B(0,r)}f(x)^2\mu(dx)\le s\int_{B(0,r)}f'(x)^2dx+c_2K_0(r) (1+s^{-1/2})\Big(\int_{B(0,r)}|f|(x)\mu(dx)\Big)^2.$$

Let $\varphi\in C^\infty(\R)$ such that $\varphi(x)=|x|$ for all $|x|\ge 1$. For $\theta\in (0,1)$, set $V(x)=\varphi(x)^\theta.$ Then, there are constants $c_3, c_4$ and $r_0>0$ such that
$$a(x) V''(x)\le -c_3\frac{a(x)}{(1+|x|)^2}V(x)+c_4\I_{B(0,r_0)}(x),\quad x\in \R.$$
This, together with the local super Poincar\'{e} inequality above, gives us the desired $\beta$ given by \eqref{thm4-3}.
\end{proof}

\medskip

\begin{acknowledgement}
The second author would like to thank Dr.\ Nikola Sandri\'{c} for
very helpful communications on \cite[Theorems 1.1 and 1.2]{Sa} and
their proofs.
\end{acknowledgement}

 \vskip 0.3truein

 {\bf Zhen-Qing Chen}

  Department of Mathematics, University of Washington, Seattle,
WA 98195, USA

  E-mail: \texttt{zqchen@uw.edu}

\bigskip

 {\bf Jian Wang}

 School of Mathematics and Computer Science, Fujian Normal
University, 350007, Fuzhou,

\quad P.R. China.

 E-mail:  \texttt{jianwang@fjnu.edu.cn}

\end{document}